\documentclass[10pt]{amsart} 

\usepackage[psamsfonts]{amssymb} 
\usepackage{amsfonts,amsmath} 
\usepackage{graphicx, color}
\usepackage{enumerate}
\usepackage{url}
\usepackage{tikz-cd}

\usepackage{amsthm, amssymb}

\usepackage[all]{xy}

\newtheorem{thmA}{Theorem}

\newtheorem{corA}[thmA]{Corollary}

\newtheorem{theorem}{Theorem}[section]
\newtheorem{prop}[theorem]{Proposition}
\newtheorem{proposition}[theorem]{Proposition}
\newtheorem{lemma}[theorem]{Lemma}
\newtheorem{corollary}[theorem] {Corollary}

\newtheorem{question}{Question}
\newtheorem*{claim*}{Claim}

\theoremstyle{remark}
\newtheorem{remark}[theorem]{Remark}

\theoremstyle{definition}
\newtheorem{definition}[theorem]{Definition}

\def\Z{\mathbb Z}

\def\stab{{\rm{Stab}}}

\def\aut{{\rm{Aut}}}
\def\out{{\rm{Out}}}

\def\gln{{\rm{GL}}(n, \Z)}

\def\gl{{\rm{GL}}}

\def\hom{{\rm{Hom}}}

\def\cd{{\rm{cd}}}

\def\<{\langle}
\def\>{\rangle}

\newcommand{\st}{\mathrm{st}}
\newcommand{\lk}{\mathrm{lk}}

\newcommand{\rar}{\rightarrow}
\def\co{\colon}
\def\ostar{\mathring\st}

\newcommand{\under}{\triangleleft}

\newcommand{\FS}{\mathcal{FS}_N}
\newcommand{\FSr}{{\mathcal{FS}^r_N}}
\newcommand{\X}{\mathcal{X}_N}
\newcommand{\Xr}{{\mathcal{X}_N^r}}
\newcommand{\CV}{\mathcal{CV}_N}
\newcommand{\CVr}{\mathcal{CV}^r_N}
\newcommand{\J}{\mathcal{J}_N}
\newcommand{\dJ}{\partial\J}
\newcommand{\dFr}{{\partial\mathcal{FS}_N^r}}
\newcommand{\dF}{\partial\mathcal{FS}_N}
\newcommand{\Lm}{\mathcal{L}_N}
\newcommand{\dL}{\partial\mathcal{L}_N}
\newcommand{\K}{\mathcal{K}_N}
\newcommand{\dK}{\partial\mathcal{K}_N}

\newcommand{\mcg}{{{\mathrm{MCG}}(\Sigma)}}
\newcommand{\cc}{\mathcal{C}(\Sigma)}

\newcommand{\vc}{^\textnormal{vc}}
\newcommand{\lf}{\textnormal{lf}}

\newcommand{\fund}[1]{\bar{#1}}		%Notation for the fundamental class.

\newcommand{\shchain}{\eta}			%Specifically cosheaf valued chains

\newcommand{\intcochain}{\psi}		%integer valued cochain
\newcommand{\smallsimplex}{\sigma}	%small simplex in double complex, or site for local (co)homology group, simplex in X.

\newcommand{\hei}{\textnormal{ht}}

\DeclareMathOperator{\im}{im}

\newcommand{\Ind}{\textnormal{Ind}}

\newcommand{\less}[1]{ _<#1}
\newcommand{\lesseq}[1]{ _\leq #1}

\title[CM Complexes, Duality Groups, and the dualizing module of ${\rm{Out}}(F_N)$]{Cohen--Macaulay Complexes, Duality Groups, and the dualizing module of ${\rm{Out}}(F_N)$}
\author{Richard D. Wade and Thomas A. Wasserman}   

\begin{document}

\begin{abstract} 
We explain how Cohen--Macaulay classifying spaces are ubiquitous among discrete groups that satisfy Bieri--Eckmann duality, and compare Bieri--Eckmann duality to duality results for Cohen--Macaulay complexes. We use this comparison to give a description of the dualizing module of ${\rm{Out}}(F_N)$ in terms of the \emph{local cohomology cosheaf} of the spine of Outer space.
\end{abstract}

\address{}
\email{}

\maketitle
\section{Introduction}

\subsection{Duality groups}

 A group $G$ is a \emph{(Bieri--Eckmann) duality group} of dimension $n$ if there exists a $G$-module $D$ called the \emph{dualizing module} and a \emph{fundamental class} $c \in H_n(G;D)$ such that for every integer $k$ and $G$--module $A$ there is an isomorphism
 \[c \cap - \colon H^k(G;A) \xrightarrow{\cong} H_{n-k}(G;A \otimes D), \]
induced by the cap product. This notion, introduced in a seminal paper of Bieri and Eckmann \cite{BE}, is a group-theoretic generalization of Poincar\'e duality. Indeed, fundamental groups of closed, aspherical manifolds are duality groups, and in this case the dualizing module $D=\mathbb{Z}$ is isomorphic to the orientation module of the fundamental group. However, duality groups form a much broader family, and are particularly prominent in low-dimensional topology and geometric group theory. Arithmetic lattices, mapping class groups, and outer automorphism groups of free groups are all, up to finite-index, duality groups \cite{BS, Harer, BF}.

 Suppose that $G$ acts freely and cocompactly on a contractible CW--complex $X$. Then the compactly supported cohomology $H_c^*(X,\mathbb{Z})$ of $X$ is isomorphic to $H^*(G,\mathbb{Z}G)$ (as a graded $G$-module).  A theorem of Bieri and Eckmann states that $G$ is a duality group of dimension $n$ if and only if $H_c^*(X,\mathbb{Z})\cong H^*(G,\mathbb{Z}G)$ is torsion-free and concentrated in degree $*=n$ (\cite[Theorem~4.5]{BE}). Furthermore, in the above situation the dualizing module satisfies $D\cong H_c^n(X,\mathbb{Z})\cong H^n(G;\Z G)$. As compactly supported cohomology is only invariant up to proper homotopy, it is desirable to work with \emph{compact} classifying spaces (equivalently, cocompact $\mathrm{EG}$s) when studying duality groups.

\subsubsection{Duality of the mapping class group via Lefschetz duality} For $\gl(N,\mathbb{Z})$ and mapping class groups, the dualizing module can be described in terms of the homology of a particularly nice $G$-complex. For $\gl(N, \mathbb{Z})$ this $G$-complex is the rational Tits building \cite{BS}, whereas for mapping class groups the dualizing module is isomorphic to the the homology of the curve complex $\mathcal{C}(\Sigma)$ \cite{Harer}. Via Bieri--Eckmann duality, such concrete descriptions of dualizing modules have enabled a wealth of cohomological calculations \cite{CFP2,CP,CFP1,MPP,PS,BMPSW,BPS}.

One goal of this article is to explain why we should not expect such a clean description for $\out(F_N)$, so for context we will take a short detour to sketch Harer's proof \cite{Harer} that the dualizing module  $D \cong H_*(\mathcal{C}(\Sigma))$ when $G=\mcg$ (see also \cite{Broaddus, Ivanov}). In short, the proof combines Lefschetz duality with a good understanding of the topology of $\mathcal{C}(\Sigma)$.

To give more details, let $\Sigma$ be a closed, orientable surface of genus $g \geq 2$. Its Teichm\"uller space $\mathcal{T}$ is homeomorphic to an open ball of dimension $6g-6$. This is not cocompact under the action of $\mcg$. However, one can pass to the \emph{thick part} $X \subset \mathcal{T}$ formed by cutting out all \emph{thin regions} $T_c$ consisting of marked surfaces where a simple closed curve $c$ is of length $<\epsilon$. These regions are contractible, and if $\epsilon$ is chosen small enough, regions overlap if and only if their defining curves are disjoint. One can use this to show that $X$ is a contractible manifold with boundary $\partial X \simeq \mathcal{C}(\Sigma)$. Harer proves that $\cc$ is homotopy equivalent to a wedge of spheres of dimension $2g-2$, which allows him to apply Lefschetz duality to see that $H^*_c(X;\mathbb{Z})\cong H_{6g-6-*}(X,\partial X;\mathbb{Z})\cong H_{6g-7-*}(\partial X ;\mathbb{Z})$ is a torsion-free group concentrated in $*=4g-5$. This gives an elegant description of the dualizing module of $\mcg$ in terms of the homology of the curve complex.

\subsubsection{The failure of Harer's approach for $\out(F_N)$}
It is at this point one begins to see the difficulty in proving analogous results for outer automorphism groups of free groups. The analogous object to $\mathcal{T}$, which is Culler--Vogtmann's Outer space $\CV$, is not a manifold, so one does not have access to Lefschetz  duality. Furthermore, although a common bordification of $\CV$ is at the heart of the duality proofs of both \cite{BF} and \cite{BSV}, the homotopy type of its boundary is unknown (see \cite{BG} for a detailed discussion). 

However, Vogtmann \cite{Vogtmann} proved that Outer space does satisfy a strong local property called the \emph{Cohen--Macaulay (CM)} property, which implies that even though $\CV$ does not satisfy Lefschetz duality, it does satisfy \emph{CM duality}. This is our second main topic of discussion.

\subsection{Combinatorial (co)sheaves, the Cohen--Macaulay property, and CM duality}

Let $X$ be a locally finite simplicial complex. We blur the destinction between $X$ and its associated poset of simplices, ordered by inclusion. A \emph{combinatorial sheaf} $\mathcal{F}$ is a covariant functor from $X$ to the category of abelian groups. A \emph{combinatorial cosheaf} $\mathcal{G}$ is a contravariant functor from $X$ to the category of abelian groups. In each of these cases, each simplex is assigned an abelian group, and there are compatible maps between these groups `going up' (with respect to inclusion of simplices) in the case of a sheaf, and `going down' in the case of a cosheaf. There is a natural way of taking (compactly supported) cohomology with coefficients in a combinatorial sheaf, and taking (locally finite) homology with coefficients in a combinatorial cosheaf. Importantly, taking local homology at simplices (in the classical sense) gives a combinatorial sheaf on $X$, and similarly local cohomology forms a combinatorial cosheaf. 

A locally finite simplicial complex is a \emph{homology CM complex} of dimension $n$ if the reduced homology $\tilde{H}_*(X;\mathbb{Z})$ is concentrated in degree $n$ and the local homology at each point is concentrated in degree $n$. If only the local condition holds, then $X$ is a \emph{local homology CM complex}. In this case, every maximal simplex will be of dimension $n$, and there is a \emph{fundamental class} $[X] \in H^{\textnormal{lf}}_n(X;h^n)$ that lives in the top-dimensional locally finite homology with coefficients in the local cohomology cosheaf $h^n$ (this is described in Section~\ref{s:sheaf_stuff}). Poincar\'e duality for non-compact manifolds splits into two statements: one can either take compactly supported cohomology and regular homology, or instead use cohomology and locally finite (or Borel--Moore) homology. In the case of duality for locally CM spaces, a further split happens where one either takes twisted coefficients with the local homology sheaf $h_n$ on the left-hand side, or the local cohomology cosheaf on the right-hand side:

\begin{thmA}[CM Duality, \cite{WW1}] \label{t:duality} Let $X$ be a locally finite, $n$-dimensional, local homology CM complex. Then there are isomorphisms 
\begin{align*}
H_c^k(X;h_n) &\to H_{n-k}(X;\mathbb{Z}) \\
H^k(X;h_n) &\to H^{\lf}_{n-k}(X;\mathbb{Z}) \\
H_c^k(X;\mathbb{Z}) &\to H_{n-k}(X;h^n) \\
H^k(X;\mathbb{Z}) &\to H^{\lf}_{n-k}(X;h^n)
\end{align*} induced by the cap product with the fundamental class $[X] \in H_n^\lf(X;h^n)$.
\end{thmA}

The proof of Theorem~\ref{t:duality} is given in \cite{WW1}, where we also discuss its relationship to related results, such as Verdier duality. For this introduction, we will limit ourselves to highlighting four salient features: the construction does not require working over a field, does not require passing to a derived category, is simplicial (hence amenable to explicit calculation), and is given by an explicit cap product. 

There is a stronger, homotopical version of the CM condition, which will take a smaller role in this paper.\footnote{The Cohen--Macaulay condition for simplicial complexes originates in Reisner's theorem in combinatorial topology, which shows that the face ring of a finite simplicial complex is Cohen--Macaulay if and only if the complex is homology CM (see Bj\"orner's survey for more on this topic \cite{Bjoerner}). Quillen's 1978 article \cite{Quillen} was influential in introducing the CM condition to geometric topologists. In it, he used a homotopical condition on links, stating: \emph{`we use the stronger definition so that our theorems are in their best form.'} Since Quillen's article, the stronger homotopy CM condition has been mostly used without comment in geometric topology. For the majority of our theorems to be in their best form, we will need to work with homology CM complexes.}

\subsection{New descriptions of dualizing modules}

If $G$ is a duality group with a locally CM classifying space, CM duality can be applied to give alternative descriptions of dualizing modules. In \cite{WW1}, we show that the cap product in CM duality is equivariant with respect to the $G$-action, therefore:

\begin{corA}
\label{c:homological_D}
If $G$ is a duality group of dimension $d$ and $G$ admits a free, cocompact action on a contractible, locally finite, homology CM complex $X$ of dimension $n$, then the dualizing module is $G$-equivariantly isomorphic to $H_{n-d}(X;h^n)$, where $h^n$ is the degree $n$ local cohomology cosheaf of $X$.
\end{corA}

This gives a homological description of the dualizing module of $G$, at the cost of using more complicated coefficients. However, this alternative description may be more amenable to calculation. 

If the space $X$ has the same dimension as $G$, then there is a  further possible description of the dualizing module in terms of \emph{global section group} $\Gamma(X)$ of the local homology sheaf. Let $\hom_c(\Gamma(X),\mathbb{Z})$ be the group of \emph{compactly determined} dual homomorphisms $f\colon \Gamma(X) \to \mathbb{Z}$. (A homomorphism $f$ is compactly determined if there exists a compact set $K$ such that $f(s)=0$ for every section $s$ that vanishes on $K$.) We say that \emph{the local homology sheaf on $X$ is semistable} if the inverse system $\{\Gamma(K)\}_K$ of sections of the local homology sheaf over finite subcomplexes satisfies the Mittag-Leffler condition (see Section~\ref{s:semistable}). Recall that while dualizing modules are always torsion-free, they are not necessarily free as abelian groups (see also Question~\ref{q:Peter}).

\begin{thmA}\label{t:compactly_sup_sections}
Let $G$, $X$ be as in Corollary~\ref{c:homological_D}. Furthermore suppose that $\cd(G)=\dim(X)$. Then the local homology sheaf on $X$ is semistable if and only if the dualizing module $D$ of $G$ is free abelian. If this is the case then \[ D \cong \hom_c(\Gamma(X),\mathbb{Z}),\] where $\Gamma(X)$ is the group of global sections of the local homology sheaf of $X$.
\end{thmA}

Corollary~\ref{c:homological_D} and Theorem~\ref{t:compactly_sup_sections} show that the dualizing module in these situations can be read off from properties of both the local homology sheaf and the local cohomology cosheaf. The work surrounding Theorem~\ref{t:compactly_sup_sections} is influenced by, and uses parts of, Geoghegan's work on semistability in \cite[Chapters 12-13]{Geoghegan}.

\subsection{Outer space and $\out(F_N)$}

Let $\CV$ and $\CVr$ be Outer space and its reduced version (consisting of graphs with no separating edges). As we briefly mentioned earlier, there are two proofs that $\out(F_N)$ is a duality group.  Both either directly (in \cite{BSV}) or indirectly (in \cite{BF}) use the topology of a bordification of $\CVr$ called \emph{Jewel space} $\J$. This is a polyhedral complex sitting inside $\CVr$ whose interior is homeomorphic to $\CVr$.

\begin{figure}[h]
	\begin{tikzcd}[column sep=large]
	*\simeq	\Xr \arrow[r,hook,"\simeq","\mbox{\begin{tiny}proper\end{tiny}}"']& \J \arrow[r,hook,"\simeq"] & \CVr \arrow[r,hook,"\simeq"] & \FSr\\
																			& \dJ \arrow[u,phantom,"\mbox{\rotatebox{90}{$\subset$}}"] \arrow[rr, "\simeq"]&&\dFr \arrow[u,phantom,"\mbox{\rotatebox{90}{$\subset$}}"]
	\end{tikzcd}
	\caption{The spine, Jewel space, Outer space, and its simplicial completion are all contractible. $\out(F_N)$ acts properly and cococompactly on $\Xr$ and $\J$, but does not act cocompactly on $\CVr$ and does not act properly on $\FSr$. The boundary of Jewel space is homotopy equivalent to $\dFr$, although this homotopy equivalence is also not proper; $\dJ$ is locally finite, whereas $\dFr$ is not.}
\end{figure}

Bux--Smillie--Vogtmann \cite{BSV} showed that $\J$ is homeomorphic to the bordification of Outer space used by Bestvina--Feighn in \cite{BF}. Vogtmann \cite{Vogtmann2} showed that $\dJ$ is homotopy equivalent to the boundary of the simplicial completion of reduced Outer space, denoted by $\dFr$ (equivalently, this space is the boundary of the free splitting complex). Jewel space can be seen as a thickening of the reduced spine $\Xr \subset \CVr$. Importantly, unlike the action on $\CV$, the action of $\out(F_N)$ on both $\J$ and $\Xr$ is cocompact.

In analogy with the mapping class group picture, $\dJ$ plays a similar role to the curve complex, and it is natural to ask about the relationship between the dualizing module of $\out(F_N)$ and the homology of $\dJ$. 

\begin{thmA} \label{t:out}
Let $\X$ be the spine of Outer space and let $\Xr$ be the spine of reduced Outer space. The dualizing module of $\out(F_N)$ is isomorphic to the groups \begin{align*} H_0(\X,h^{2N-3})&\cong H^{2N-3}_c(\X,\mathbb{Z}) , \text{ and}\\
H_0(\Xr,h_r^{2N-3}) &\cong H^{2N-3}_c(\Xr,\mathbb{Z}) , \end{align*} where $h^{2N-3}$ and $h_r^{2N-3}$ are the degree $2N-3$ local cohomology cosheaves of $\X$ and $\Xr$, respectively. In contrast, the homology of the simplicial boundary of (reduced) Outer space satisfies
\begin{align*} H_{3N-5-k}(\dF;\mathbb{Z}) &\cong H^k_c(\X;h_*(\CV)|_{\X}),\text{ and}\\ H_{3N-5-k}(\dFr;\mathbb{Z}) &\cong H^k_c(\X;h_*(\CVr)|_{\Xr}), \end{align*} where $h^*(\CV)|_{\X}$ and and $h^*(\CVr)|_{\Xr}$  are the restrictions of the local homology sheaf of Outer space (and reduced Outer space) to the spine (and its reduced version, respectively).
\end{thmA}

Given that $\dJ \simeq \dFr$ \cite[Proposition~2.14]{Vogtmann2}, this theorem shows that the difference between the homology of $\dJ$ and the dualizing module is the difference between taking compactly supported cohomology of $\Xr$ with coefficients in $h_*(\CVr)|_{\Xr}$ and $\mathbb{Z}$. Unfortunately there is no simple `universal coefficient theorem' in this setting: such a result would be desirable as it would allow us to make this relationship between the homology of $\dJ$ and the dualizing module precise. Theorem~\ref{t:out} also provides a homological description of the dualizing module, at the additional cost of using coefficients in the local cohomology cosheaf.

Using Vogtmann's proof that the (reduced) spine is homotopy CM, in Theorem~\ref{t:Jewel_space_is_CM} we show that a natural thickening of the (reduced) spine is homotopy CM. Using work of Br\"uck--Gupta \cite{BG}, we then show that the boundary of the thickened spine is homotopy equivalent to the boundary of the free splitting complex. With these results in hand, Theorem~\ref{t:out} is a consequence of CM duality. The first part follows from the fact that the spine is locally CM \cite{Vogtmann} and the third isomorphism in Theorem~\ref{t:duality}. The second part follows from a relative version of CM duality. 

\subsection{A bound for the top-dimensional cohomology of $\out(F_N)$}

Bieri--Eckmann duality allows one to translate problems from high-dimensional cohomology to low-dimensional homology, at the cost of understanding the dualizing module. If $G$ is the mapping class group or $\gln$ and $d$ is the virtual cohomological dimension of $G$, then $H^d(G;\mathbb{Q})=0$. This is a result of Church--Farb--Putman \cite{CFP2} in the case of the mapping class group, and Lee and Szczarba in the case of $\gln$ \cite{LS}, and is obtained by a vanishing result for the coinvariants of the associated dualizing module. In the case of $G=\gln$, the rational cohomology of $G$ also vanishes in codimensions one and two \cite{CP,BMPSW}. The analogous codimension-one vanishing result was shown not to hold in the mapping class group in work of Chan--Galatius--Payne \cite{CGP}. 

Infamously, one cannot prove any such vanishing result for $\out(F_N)$: the vcd of $\out(F_7)$ is 11 and Bartholdi \cite{Bartholdi} showed that $H^{11}(\out(F_7);\mathbb{Q}) \cong \mathbb{Q}$. The proof is computer-based, and it is not clear if there is a geometric origin for this cohomology class. While we cannot solve this problem, the methods in this paper allow us to obtain a resolution of the dualizing module of $\out(F_N)$ via the chain complex for $H_*(\Xr,h^{2N-3}_r)$. We describe this in Section~\ref{s:bound}, and use it to prove the following:

\begin{thmA}
Let $\rho \in \Xr$ be a rose in the spine of reduced outer space, and let $\stab(\rho) <\out(F_N)$ be the stabilizer of this rose. Then \[ \dim(H^{2N-3}(\out(F_N);\mathbb{Q})) \leq \dim((h^{2N-3}(\rho)\otimes \mathbb{Q})_{\stab(\rho)}),\]
where $h^{2N-3}(\rho)$ is the local cohomology of the reduced spine at this rose.
\end{thmA}

A detailed investigation of $h^{2N-3}(\rho)$ as a $\stab(\rho)$-representation is a topic for future research.

\subsection{The CM condition does not imply Bieri--Eckmann duality}

$\out(F_N)$ is not the only example of a duality group with a CM classifying space: in fact, CM complexes are ubiquitous among classifying spaces of duality groups. For instance, arithmetic groups and mapping class groups of surfaces are fundamental groups of orbifolds, and Brady and Meier proved that a right-angled Artin group (RAAG) is a duality group if and only if its associated Salvetti complex is a local homology CM complex \cite[Theorem~C and Lemma~5.1]{BM}. From this, one might hope that every fundamental group of a finite, closed  (i.e. no free faces), aspherical, local homology CM complex is a duality group. However, we will see that this is not the case.

\subsubsection{An example of Atanasov}One does not need to leave the world of RAAGs to find an aspherical CM complex $Y$ such that $\pi_1(Y)$ is not a duality group. Atanasov \cite{Atanasov} studied the two dimensional situation, and found an aspherical local-homotopy-CM complex (see Figure~\ref{fig:Atanasov}) with fundamental group $\mathbb{Z}^2 \ast \mathbb{Z}^2$ which is not a duality group. 

Using a version of CM duality, Atanasov \cite{Atanasov} gave sufficient conditions for fundamental groups of acyclic 2-complexes to be duality groups, although they require the singular parts of the 2-complex to be quite well-behaved (for instance a generic one relator group is a 2-dimensional duality group, but its presentation complex is not covered by Atanasov's criterion). 

\begin{figure}[ht]
    \centering
    \includegraphics{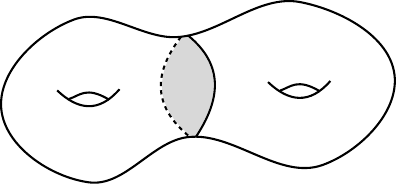}
    \caption{    The group $G=\mathbb{Z}^2  \ast \mathbb{Z}^2$ is not a duality group but is the fundamental group of an aspherical local homotopy CM complex $Y$ which has no free faces. The space $Y$ is given by gluing two tori $T_{1,1}$ each with one boundary to a disc.}
    \label{fig:Atanasov}
\end{figure}

\subsubsection{The CM property and visible irreducibility}

Atanasov's complex $Y$ given in Figure~\ref{fig:Atanasov} satisfies slightly stronger conditions than having no free faces: the complex is \emph{visibly irreducible} in the sense of Louder--Wilton (cf. \cite[Section~4]{Wilton}, or Section~\ref{s:examples2} below).  This is a condition for a complex to have no `obvious' simplifications, and plays a key role in Louder--Wilton's proof that that one-relator groups with negative immersions are coherent \cite{LW}. In Theorem~\ref{t:visible_irreducibility}, we show that visible irreduciblity can be described in terms of the local homology sheaf.

\begin{thmA}\label{t:visible_irreducibility}
Let $Y$ be a 2-complex. The following are equivalent:
\begin{itemize}
\item $Y$ is visibly irreducible.
\item $Y$ is a local homology CM complex  and face maps of the local homology sheaf are surjective.
\item $Y$ is a local homology CM complex and face maps of the local cohomology cosheaf are injective.
\end{itemize}
\end{thmA}

Theorem~\ref{t:visible_irreducibility}, combined with Atanasov's example, shows that even with strong restrictions on the local homology sheaf, fundamental groups of aspherical CM complexes may not be duality groups.

\subsection{Further Questions}

\begin{question}\label{CMclassquestion}
Does every duality group of finite type have a classifying space that is a local homology CM complex?
\end{question}
This fits in a series of questions going back to Wall \cite{Wall}, who asked whether every Poincar\'e duality group admits a closed manifold as its classifying space. That question was answered in the negative by Davis \cite{Davis1998}: there are Poincar\'e duality groups of cohomological dimension $\geq 4$ that are not finitely presented. Davis \cite[Question~3.5]{Davis2000} then asked if a finitely-presented, torsion-free group satisfies Poincar\'e duality over a ring $R$, does that group admit a closed $R$-homology manifold as its classifying space? This question has a negative answer when $R=\mathbb{Q}$ \cite{Fowler}, but is open for $R=\Z$. Note that a negative answer to this question would not imply a negative answer to Question~\ref{CMclassquestion}.

\begin{question}\label{q:duality_when}
Suppose $Y$ is a finite, aspherical, local homology CM simplicial complex with fundamental group $G$. What conditions on the local homology sheaf of $Y$ (or alternatively, its universal cover $X=\tilde{Y}$) ensure that $G$ is a duality group?
\end{question}

Given the example in Figure~\ref{fig:Atanasov}, it is likely that global conditions on the universal cover are required, as one has to relate the local properties of the space to the behaviour of $H^*_c(X;\mathbb{Z})$: CAT(0) geometry was essential in Brady--Meier's classification of when RAAGs are duality groups, and discrete Morse theory was the local-to-global tool at the heart of both proofs that $\out(F_N)$ is a duality group. 

When $G$ is a duality group one can compare Bieri--Eckmann duality and CM duality via the isomorphisms
\begin{equation*}
	\begin{tikzcd}
	H_{n-k}(G; D) &	H^k(G;\Z) \arrow[l,"c\cap -"',"\cong"]\arrow[r,"{[Y]} \cap -","\cong"'] &	H_{n-k}(Y;h^*_Y).
	\end{tikzcd}
\end{equation*}
Indeed, if $Y$ is CM then with a little work one can show that there exists a natural map
	\begin{equation*}\label{e:compareduality}
	 \pi\colon H_{k}(G; H^*_c(X;\Z))\rar 	H_{k}(Y;h^*_Y),
	\end{equation*}	
that is an isomorphism if and only if $G$ satisfies Bieri--Eckmann duality for the trivial module.

In light of Theorem~\ref{t:compactly_sup_sections}, one can as the following:

\begin{question}
Is the dualizing module of $\out(F_N)$ free abelian? More generally, is the local homology sheaf on the spine of reduced Outer space semistable?
\end{question}

We suspect one might be able to obtain a positive answer to both of these questions via careful use of the proofs in \cite{BF} or \cite{BSV}. More generally, we include the following question of Peter Kropholler:

\begin{question}[P. Kropholler]\label{q:Peter}
Is there an example of a Bieri--Eckmann duality group whose dualizing module is not free abelian?
\end{question}

By Exercise~4 of \cite[Chapter VII.10]{Brown}, when $D$ is free abelian we have $H_k(G;M) \cong H^{n-k}(G;\hom(D,M))$ for all $M$ (i.e., one can also twist coefficients on the side of cohomology).

\subsection{The structure of this article}

In Section~\ref{s:sheaf_stuff} we give an introduction to combinatorial sheaves and cosheaves and summarise the key results concerning CM duality. We describe how Corollary~\ref{c:homological_D} and Theorem~\ref{t:compactly_sup_sections} follow from \cite{WW1}. 

In Section~\ref{s:Outer_space} we introduce thickenings $\Lm$ and $\K$ of the spine and reduced spine of Outer space, show that they are homotopy CM, and use this to prove Theorem~\ref{t:out}.

In Section~\ref{s:examples2} we prove Theorem~\ref{t:visible_irreducibility} on the relationship between the local homology sheaf and the visible irreducibility condition of Louder--Wilton.

\subsection{Acknowledgements}
We thank Alexandra Pettet for asking about the possibility of translating Harer's duality proof from the mapping class group to $\out(F_N)$. We thank Rachael Boyd and Karen Vogtmann for sharing very helpful notes and references, and Peter Kropholler and Peter Patzt for helpful discussions. Both authors are supported by the Royal Society of Great Britain.

\section{Sheaf cohomology and cosheaf homology on simplicial complexes}\label{s:sheaf_stuff}

This section summarises the combinatorial sheaf theory needed for this paper; this is described in full in \cite{WW1}. The main objects of interest to us in this paper are discrete groups (usually denoted $G$) and locally finite simplicial complexes (usually denoted $X$). We will also assume $X$ is \emph{oriented}, so the vertices of each simplex have a preferred ordering. In particular there is a consistent notion of the \emph{front} and \emph{back} $k$-dimensional faces of simplices.

\subsection{Combinatorial sheaves and cosheaves}

 The simplices of a complex form a poset with respect to inclusion, where $\sigma \leq \tau$ if $\sigma$ is contained in $\tau$. A \emph{combinatorial sheaf}  on $X$ is a functor $\mathcal{F}$ from the poset of simplicies to the category of abelian groups.  In particular there is an abelian group $\mathcal{F}(\sigma)$ assigned to each simplex and a homomorphism \[ \mathcal{F}(\sigma <\tau) \colon \mathcal{F}(\sigma) \to \mathcal{F}(\tau) \] whenever $\sigma < \tau$. Requiring $\mathcal{F}$ to be a functor is equivalent to the condition that \[ \mathcal{F}(\rho < \tau) = \mathcal{F}(\sigma <\tau) \circ \mathcal{F}(\rho<\sigma)\] whenever $\rho < \sigma < \tau$. A \emph{combinatorial cosheaf} does the same thing except the maps go in the opposite direction; more precisely, a combinatorial cosheaf is a contravariant functor from the poset of simplices of $X$ to the category of abelian groups. 
 
 One can study sheaves and cosheaves on X valued in other categories, for instance $R$--modules. However in this paper, whenever we talk about a (co)sheaf, it will be valued in the category of abelian groups.

\subsection{The local homology sheaf and the local cohomology cosheaf.} \label{s:local_hom}

Let $\ostar(\sigma)$ be the neighbourhood of $\sigma$ spanned by the interiors of the simplices $\tau \geq \sigma$ (including the interior of $\sigma$ itself). Then $X - \ostar(\sigma)$ is a subcomplex of $X$. Working over $\mathbb{Z}$, let $C^{\sigma}_\bullet(X)=C_\bullet(X,X-\ostar(\sigma))=C_\bullet(X)/C_\bullet(X-\ostar(\sigma))$ be the associated relative chain complex. The \emph{$k$th local homology of $X$ at $\sigma$} is \[ h_k(\sigma):=H_k(C_\bullet^{\sigma}(X)).\] As $C^{\sigma}_\bullet(X)$ is spanned by the simplices $\tau \geq\sigma$, this gives a simple combinatorial description of $h_k(\sigma)$.

If $\sigma \leq \tau$ then we can view $C^{\tau}_\bullet(X)$ as a quotient complex of $C^{\sigma}_\bullet(X)$, which induces a map \[ h_k(\sigma <\tau) :h_k(\sigma) \to h_k(\tau), \]  induced by the associated map of pairs $(X,X-\ostar(\sigma)) \to (X,X-\ostar(\tau))$. Note that if $\sigma$ is maximal then $$C^{\sigma}_\bullet(X)= \langle \sigma \rangle$$ and $h_k(\sigma)=\mathbb{Z}$ in the dimension of $\sigma$ and is trivial in every other degree.

We say that $X$ is a \emph{local homology CM complex of dimension $n$} if $h_k(\sigma)=0$ unless $k=n$ for all $\sigma \in X$. This implies that every maximal simplex is $n$-dimensional, and as there are no $n+1$ simplices we have a simpler description \[ h_n(\sigma)= \ker (C^{\sigma}_n(X) \to C_{n-1}^{\sigma}(X))\] of the $n$th local homology. In particular, when $X$ is locally CM all the local homology groups are free abelian.

Taking $C^\bullet_\sigma(X)=\hom(C_\bullet^\sigma(X),\mathbb{Z})$ gives dual cochain complexes which compute local cohomology. This gives a (graded) cosheaf $h^*$ on $X$, with corestriction maps $h^\ast(\tau >\sigma) :h^\ast(\tau) \to h^\ast(\sigma)$. Typically for us $X$ will be a local homology CM complex and we will only be interested in the degree $n$ parts $h_n$ and $h^n$.

\subsection{Sheaf cohomology and cosheaf homology}

Let $\mathcal{F}$ be a combinatorial sheaf on $X$ and let $X_k$ be the set of $k$-simplices of $X$. We define \[C^k(X;\mathcal{F}):= \prod_{\sigma \in X_k} \mathcal{F}(\sigma).\] Let $|\sigma:\tau|=(-1)^i$ if  $\sigma$ is the $i$th facet of $\tau$, and set $|\sigma:\tau|=0$ otherwise. If $\alpha\in \mathcal{F}(\sigma)$ and $\sigma \in X_k$, we define a coboundary map termwise on $C^\bullet(X;\mathcal{F})$ by \begin{equation*}\delta\alpha= \sum_{\tau \in X_{k+1}} |\sigma : \tau| \mathcal{F}(\sigma < \tau)(\alpha).\end{equation*} One can verify that this gives a cochain complex whose cohomology we define to be the \emph{cohomology of $X$ with coefficients in $\mathcal{F}$}. When $X$ is locally finite, there is a compactly supported version of this where \begin{equation}\label{eq:cpltlysuppsheafcoh}C_c^k(X;\mathcal{F}):= \bigoplus_{\sigma \in X_k} \mathcal{F}(\sigma),\end{equation} with the coboundary maps defined the same way. 

Similarly, the \emph{cosheaf homology} of a cosheaf $\mathcal{G}$ is the homology of the chain complex given by \[C_k(X;\mathcal{G}):= \bigoplus_{\sigma \in X_k} \mathcal{G}(\sigma)\] and a boundary operator defined by \[\partial \alpha=\sum_{i=0}^k(-1)^i\mathcal{G}(\sigma > \sigma_i)(\alpha)\] for $\alpha \in \mathcal{G}(\sigma)$ and $\sigma \in X_k$, with $\sigma_i$ denoting the $i$th face of $\sigma$. When $X$ is locally finite, we can also define \emph{locally finite homology with coefficients in $\mathcal{G}$} by taking homology of the chain complex with terms \[C_k^\lf(X;\mathcal{G}):= \prod_{\sigma \in X_k} \mathcal{G}(\sigma),\] and the same boundary operator as above.

\subsection{Cap products and the fundamental class}

Suppose that $X$ is $n$-dimensional and locally finite. Then $h^n(\sigma)=\langle \sigma^*\rangle$ for every $n$-simplex $\sigma$, and the formal sum $$[X]:=\sum_{\sigma \in X_n}\sigma^*$$ determines an element of $H^{\lf}_n(X,h^n)$ called the \emph{fundamental class} \cite[Section~4.2.1]{WW1}. If $X$ is connected and locally CM, then $H^{\lf}_n(X,h^n)=\mathbb{Z}$ and is generated by the fundamental class. In \cite{WW1}, we give formulations of the cap product required for all permutations of coefficients given in Theorem~\ref{t:duality}, however in this paper we will only give the cap product
\begin{equation*}
    H_k^{\textnormal{lf}}(X,h^n) \otimes H^l_c(X,\Z) \xrightarrow{\cap} H_{k-l}(X,h^n),
\end{equation*}
induced by the map
\begin{equation}\label{cupversion2}
	\shchain \cap \intcochain= \sum_{\smallsimplex \in X_k}\intcochain(\smallsimplex_{\geq k-l})h^n( \smallsimplex> \smallsimplex_{\leq k-l})(\shchain_\smallsimplex),	
\end{equation}
where $\shchain \in C_k^{\textnormal{lf}}(X,h^n)$, $\intcochain \in C^l_c(X,\Z)$, and $\eta_\sigma\in h^n(\sigma)$ is the value of $\eta$ at $\sigma$. If $\sigma=[v_0,\ldots,v_k]$ then $\sigma_{\geq k-l}=[v_{k-l}, \ldots, v_k]$ is the \emph{back $l$--face} and $\sigma_{\leq k-l}=[v_0, \ldots, v_{k-l}]$ is the \emph{front ($k-l$)--face} of $\sigma$. 

When $\shchain=[X]$ is the fundamental class,  $\shchain_\smallsimplex=\sigma^*$ for every $n$-simplex $\sigma$ and (\ref{cupversion2}) can be unpacked further (see \cite[Section~4.2.2]{WW1} for details).

\subsection{Relative versions of CM duality} \label{s:relative_CM}

Let $X$ be a locally finite simplicial complex and $L$ be a full subcomplex of $X$, and let $L^\textnormal{vc}$ be the full complex spanned by vertices in $X - L$ (we call $L^\textnormal{vc}$ the \emph{vertex complement} of $L$). The cap product defined in the above setting applies for pairs $(X,L)$, as long as in any simplex $\sigma$, the orientation of $X$ orders vertices in $L$ before those in $L^\textnormal{vc}$. This is so that back faces, when it is appropriate, lie in $L^\textnormal{vc}$ (see \cite[Section~2.5.4]{WW1} for details).

\begin{theorem}[Relative CM Duality]\label{t:dualitytheorem}
	Let $X$ be locally finite homology CM complex of dimension $n$ with local homology sheaf $h_*$ and local cohomology sheaf $h^*$. Let $L$ be a full subcomplex of $X$ and orient $L\vc$ before $L$. Then capping with the fundamental induces isomorphisms
	\begin{align}
H^l_c(L; h_n|_L) &\xrightarrow{[\fund{X}]\cap -} H_{n-l}(X,L^\textnormal{vc};\Z), \text{ and} \\
H^l_c(X,L;\Z) &\xrightarrow{[\fund{X}]\cap -} H_{n-l}(L^\textnormal{vc};h^n|_{L^\textnormal{vc}}).
	\end{align}
	
\end{theorem}

There are other versions of these isomorphisms given in Theorem~4.26 of \cite{WW1}; we only give two relevant ones here.

\subsection{Semistability and the global section group} \label{s:semistable}

The subcomplexes of $X$ form a direct system with respect to inclusion. If $K$ is a subcomplex, we use \[ \Gamma(K)=H^0(K;h_*|_K)\] to denote the group of sections of the local homology sheaf of $X$ over $K$. If $K \subset L$, then restriction of sections gives a restriction homomorphism $r_{LK} :\Gamma(L) \to \Gamma(K)$. Local finiteness of $X$ implies that we can restrict our attention to sections over finite subcomplexes: these form an inverse system whose limit is still the global section group $\Gamma(X)$.

An inverse system $\{A_i\}$ with maps $\{f_{ij}: A_i \to A_j\}$ indexed by a filtered poset $I$ is \emph{semistable}  if for all $i \in I$ there exists $j \geq i$ such that for all $k \geq j \geq i$  the map $f_{ki}$ has the same image in $A_i$ independently of $k$. This is also known as the \emph{Mittag-Leffler} condition. 

\begin{definition}[Semistability of the local homology sheaf]We say that the local homology sheaf on $X$ is \emph{semistable} if the inverse system of sections of $h_*$ over finite subcomplexes is semistable. \end{definition}

If $A$ is an abelian group, we say a homomorphism $f\co \Gamma(X) \to A$ is \emph{compactly determined} if there exists a finite subcomplex $K \subset X$ such that $f(s)=0$ whenever a section $s \in \Gamma(X)$ vanishes on $K$ (this implies the image of a section under $f$ is determined by its restriction to $K$). We use $\hom_c(\Gamma(X),A)$ to denote the group of compactly determined homomorphisms from $\Gamma(X)$ to $A$. Using work of Geoghegan \cite{Geoghegan}, in \cite{WW1} we prove the following:

\begin{theorem}[\cite{WW1}, Corollary~3.19]
Suppose that $X$ is locally CM. Then the local homology sheaf on $X$ is semistable if and only if $H_0(X;h^*)$ is a free abelian group. If this is the case, then \[H_0(X;h^*) \cong \hom_c(\Gamma(X),\mathbb{Z}).\] 
\end{theorem}

This allows us to prove Theorem~\ref{t:compactly_sup_sections} from the introduction.

\begin{theorem}[Theorem~\ref{t:compactly_sup_sections}]
Let $G$ be a duality group that admits a free, cocompact action on a contractible, locally finite, homology CM complex $X$ of dimension equal to $\cd(G)$. Then the local homology sheaf on $X$ is semistable if and only if the dualizing module $D$ of $G$ is free abelian. If this is the case then \[ D \cong \hom_c(\Gamma(X),\mathbb{Z}),\] where $\Gamma(X)$ is the group of global sections of the local homology sheaf of $X$.
\end{theorem}

\begin{proof}
When the dimension of $X$ is the same as the cohomological dimension of $G$, under the above hypothesis we have $D \cong H_0(X;h^*)$ by CM duality (Theorem~\ref{t:duality}). 
\end{proof}

\section{The dualizing module of $\out(F_N)$} \label{s:Outer_space}

In this section we discuss the dualizing module of $\out(F_N)$ and its relationship with the spine of Outer space, and a combinatorial version of Bux--Smillie--Vogtmann's \emph{Jewel space}.

\subsection{Poset preliminaries} \label{s:poset_preliminaries}

We have already defined combinatorial sheaves and cosheaves using the \emph{simplex poset} of a simplicial complex, which is the partially ordered set of simplices with respect to inclusion. Conversely, every poset $P$ has an associated simplicial complex $|P|$ called the \emph{geometric realization}. The  complex $|P|$ has vertex set $P$ and a $k$-simplex $\sigma$ for every chain $p_0 < p_1 < \cdots < p_k$ of $k+1$ elements in $P$. As a result, it is convenient to sometimes blur the distinction between a poset and its associated simplicial complex (for example, we will talk about the homotopy type of a poset, meaning the homotopy type of its geometric realization). A map $f:P \to Q$ between posets is a \emph{poset map} if it is order-preserving.

We use $P'$ to denote the \emph{poset of finite chains} in $P$, where each element of $P'$ is a finite length chain in $P$ and $c_1 < c_2$ if $c_2$ is a refinement of $c_1$. The geometric realization of $P'$ is the barycentric subdivision of the geometric realization of $P$. This gives a convenient way to work with barycentric subdivisions. 

For $p,p' \in P$ we define the subposets \begin{align*}p_< &:= \{ q \in P : p < q \} \\
\less{p} &:= \{ q \in P : q < p \} \\ (p,p') &:= \{ q \in P : p < q <p'\}, \end{align*}

and $p_\leq$, $\lesseq{p}$ are defined similarly. The \emph{join} $P \ast Q$ of two posets has underlying set $P \cup Q$ with partial orders on $P$ and $Q$ given by restriction, with the additional relation $p <q$ for all $p \in P$, $q \in Q$. If $\sigma= p_0 <p_1 < \cdots < p_k$ is a $k$-simplex in $|P|$, then its link is given by the join \begin{equation} \less{(p_0)} \ast (p_0,p_1) \ast \cdots \ast (p_{k-1},p_k) \ast (p_k)_< \label{e:join}\end{equation}

We use $\hei_P(p)$ to denote the \emph{height} of an element $p \in P$, dropping the subscript where the overlying poset is understood. This is defined to be the dimension of $\lesseq{p}$, or equivalently the length of the longest chain in $\less{p}$. 

\begin{definition}[Homotopy CM complexes]
A simplicial complex $X$ is an \emph{$n$-dimensional homotopy CM complex} if it is contractible or homotopy equivalent to a wedge of $n$-spheres, and the link $\lk_X(\sigma)$ of every $k$-simplex $\sigma \in X$ is contractible or homotopy equivalent to a wedge product of a set of $(n-k-1)$--spheres. We say a poset is homotopy CM if its geometric realization is.
\end{definition}

A homotopy CM complex is also homology CM. Using \eqref{e:join}, one can obtain the following characterisation of the homotopy CM property (we learned this from Section~2 of \cite{Vogtmann}, where it appears with slightly different notation).

\begin{proposition}[\cite{Quillen}] \label{p:homotopy_CM_conditions}
A poset $P$ is $n$-homotopy CM if and only if:
\begin{itemize}
    \item $P$ is $n$-spherical,
    \item $\less{p}$ is $(\hei(p)-1)$-spherical for all $p \in P$,
    \item $(p,p')$ is $(\hei(p)-\hei(p')-2)$-spherical whenever $p<p'$ in $P$, and
    \item $p_<$ is $(n-\hei(p)-1)$-spherical for all $p \in P$.
\end{itemize}
\end{proposition}

We will also need \emph{Quillen's Fiber Lemma} (\cite{Walker} has a concise proof).

\begin{theorem}[Quillen's Fiber Lemma \cite{Quillen}]\label{t:quillen}Let $f \colon P \to Q$ be a poset map. If $f^{-1}(\lesseq{q})$  is contractible for all $q \in Q$, then $f$ is a homotopy equivalence. If $f^{-1}(q_\leq)$ is contractible for all $q \in Q$, then $f$ is a homotopy equivalence.
\end{theorem}

\begin{corollary}[Monotone retractions are homotopy equivalences]\label{c:quillen}
Let $r:P \to P$ be a monotone increasing or decreasing (either $r(p) \leq p$ for all $p$ or $r(p) \geq p$ for all $p$) poset map such that $r^2=r$. Then $r$ induces a homotopy equivalence between $P$ and $r(P)$.
\end{corollary}

\begin{proof}
Let us only consider the case that $r(p) \leq p$ for all $p$. If $q \in r(P)$ then $r(q)=q$. If $q\leq r(p)$ then $q \leq p$. Hence $r^{-1}(q_\leq)$ has $q$ as a minimal element, so is contractible. Hence $r$ is a homotopy equivalence by the Fiber Lemma.
\end{proof}

We call a poset map $r$ as in Corollary~\ref{c:quillen} a \emph{retraction}, for obvious reasons. The above corollary holds under weaker conditions: the map $r$ can be increasing at some points and decreasing at others and $r^2=r$ is not needed (see \cite{Quillen, Bjo2}).

\subsection{Thickening the spine of Outer space}

A \emph{free splitting} of $F_N$ is a minimal action of $F_N$ on a simplicial tree with trivial edge stabilizers. Two splittings are equivalent if they are $F_N$-equivariantly homeomorphic. One can also think of free splittings as marked graphs-of-groups with trivial edge groups. If $T,T'$ are free splittings we say that $T < T'$ if $T$ can be obtained from $T'$ by collapsing a forest. With this partial order, the set of free splittings forms a poset $\FS$ whose geometric realization is called the \emph{free splitting complex}.  The \emph{reduced free splitting complex} $\FSr$ is the full subcomplex of $\FS$ spanned by trees $T$ such that $T/F_N$ contains no separating edges.

We use $\X \subset \FS$ to denote the \emph{spine of Outer space}: this is the full subcomplex of $\FS$ spanned by splittings where the action of $F_N$ is free (equivalently, $T/F_N$ has fundamental group $F_N$). The \emph{reduced spine} $\Xr$ is defined similarly. We define $\dF$ and $\dFr$ to be the full subcomplexes of $\FS$ and $\FSr$ that are spanned by trees $T$ such that the action of $F_N$ on $T$ is \emph{not} free. Equivalently, these are the full subcomplexes of $\FS$ and $\FSr$ spanned by vertices that are not in the spine.

Let $\FS'$ be the poset of chains (i.e. barycentric subdivision) of $\FS$. Vertices in $\FS'$ are given by chains $c=T_1<T_2< \cdots < T_k$ of elements of $\FS$, and $k$-simplices correspond to chains of chains $c_0 < c_1 < \cdots < c_k$, where the partial order on chains is given by inclusion. For a chain $c \in \mathcal{FS}'$ let $\min(c)$ and $\max(c)$ be the first and last elements of the chain, respectively.

\begin{figure}
    \centering
    \begin{minipage}{0.45\textwidth}
        \centering
        \includegraphics[width=0.9\textwidth]{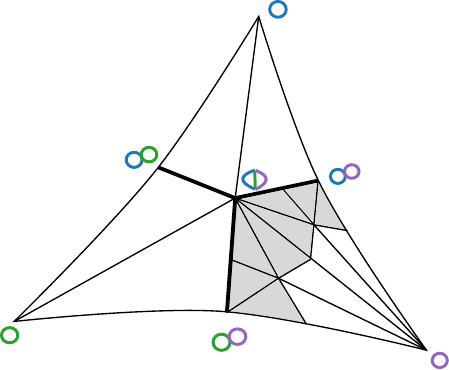} % first figure itself
        
    \end{minipage}\hfill
    \begin{minipage}{0.45\textwidth}
        \centering
        \includegraphics[width=0.9\textwidth]{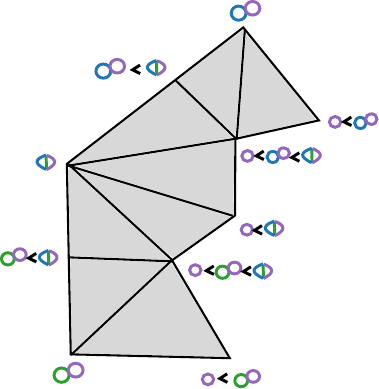} % second figure itself
        
    \end{minipage}

    \caption{Left: a portion of the free splitting complex $\mathcal{FS}_2$ in rank two. The intersection with the spine $\mathcal{X}_2$ is the (thickened) central tripod. Part of the barycentric subdivision $\mathcal{FS}_2'$ is depicted in the bottom corner, where the thickened spine  $\K$ is shaded.
    Right: an enlarged version of the featured part of $\K$. Vertices are given by chains of vertices in the free splitting complex, and simplices correspond to increasing chains of chains. The boundary $\dK$ can be seen in the picture in the chains with single circle as a minimal element.}
    \label{fig:enter-label}

\end{figure}

\begin{definition}[Thickened spines]Let $\Lm$ be the subcomplex of $\FS'$ spanned by chains $c$ such that $\max(c) \in \X$. Let $\K$ be the subcomplex of $\FSr'$ spanned by chains $c$ such that $\max(c) \in \Xr$. We call $\Lm$ and $\K$ \emph{thickened spines}.
\end{definition}

The spaces $\Lm$ and $\K$ are combinatorial thickenings of the spine of Outer space and reduced Outer space respectively, which fulfill the same role as Jewel space (see Remark~\ref{r:jewel} below). Such an approach also appears in work of Br\"uck--Gupta \cite{BG}, who studied the homotopy type of the boundaries of these complexes.

\begin{proposition} \label{p:retract}
There is a monotone map $\chi \colon \Lm \to \Lm$ with image $\X'$ given by taking $c \in \Lm$ to the subchain $\chi(c)$ of $c$ consisting of trees in $\X$. The spine $\X$ is contractible by the main theorem of \cite{CV}, therefore so is $\Lm$. The map $\chi$ restricts to a monotone map of $\K$ with image $\Xr'$, so $\K$ is also contractible by \cite{CV}.
\end{proposition}

\begin{proof}
If $c' \in \Lm$ is a subchain of $c \in \Lm$ then $\chi(c')$ is a subchain of $\chi(c)$ so $\chi$ is a poset map that restricts to the identity on $\X'$. Hence $\chi$ is monotone, and in particular is a homotopy equivalence. The same argument applies to $\K$, and both $\X$ and $\Xr$ are shown to be contractible in \cite{CV}.
\end{proof}

\begin{definition}[Boundaries of the thickened complexes]
Let $\dL$ be the full subcomplex of $\Lm$ spanned by chains $c$ such that $\min(c)\in \dF$. Let $\dK$ be the full subcomplex of $\K$ spanned by chains $c$ such that $\min(c) \in \dFr$.
\end{definition}

See Figure~\ref{fig:enter-label} for visual aid. The following result appears in \cite{BG} in a slightly different guise -- we sketch how their results apply in this setting below.

\begin{theorem}\label{t:bg}
The geometric realization of $\dL$ is homotopy equivalent to $\dF$. The geometric realization of $\dK$ is homotopy equivalent to $\dFr$.
\end{theorem}

\begin{proof}[Sketch proof]
We apply the Fiber Lemma. There is a poset map  $\phi \colon \dL \to \dF'$ given by taking a chain $c \in \dL$ to the subchain spanned by trees in $\dF$. Let $c=T_0<T_1<\cdots<T_k$ be a chain in $\dF'$. Every point in  $d \in \phi^{-1}(c)$ is obtained by adding a nontrivial chain $T_{k+1}<T_{k+2} < \cdots < T_l$ to $c$ with each $T_i \in \X$. Points in $\phi^{-1}(c_\leq)$ are obtained by taking such a $d$ and making a further refinement $d'$ to possibly include more trees in $\dF$. The map $d' \to d$ given by forgetting the additional points in $\dF$ is a monotone retraction that induces a homotopy equivalence between $\phi^{-1}(c_\leq)$ and $\phi^{-1}(c)$ (see also Lemma~7.1(1) of \cite{BG}). 

Let $i \geq k+1$. The graph $T_k/F_N$ is obtained from $T_i /F_N$ by collapsing a subgraph. As $T_i \in \X$ and $T_k \not \in \X$, at least one component of this subgraph has nontrivial fundamental group. Therefore there exists $T_k<T_i'<T_i$ such that $T_i'$ is a \emph{core extension} of $T_k$ -- the collapsed edges of $T_i'/F_N$ form a core graph (a graph with no valence one vertices). If each $T_i'$ is chosen maximally, the map taking the chain $T_{k+1}<T_{k+2}< \cdots < T_l$ to $T_{k+1}'<T_{k+2}'< \cdots < T_l'$ gives a monotone retraction from $\phi^{-1}(c)$ to the poset of core extensions of $T_k$. This is shown to be contractible in Proposition~7.2 of \cite{BG}. Therefore $\phi$ is a homotopy equivalence by the Fiber Lemma.

The map $\phi$ restricts to a map $\psi \colon \dK \to \dFr'$. The same argument as above shows that this is a homotopy equivalence, using Lemma~7.1(3) and Proposition~7.3(2)  of \cite{BG} in place of Lemma~7.1(1) and Proposition~7.2 of \cite{BG}.

\end{proof}

\begin{remark}\label{r:jewel}
The complex $\K$ is a variation on the Jewel space introduced by Bux--Smillie--Vogtmann. The original Jewel space is made out of larger polytopes rather than simplices. The combinatorial version  appearing here is convenient for arguments involving posets, but one can see from the figures that taking the barycentric subdivision requires using many more simplices compared to the polyhedral jewels appearing in, e.g. Figure~1 and Figure~2 of \cite{BSV}. The boundary $\dJ$ of Jewel space is homotopy equivalent to $\dFr$ by \cite{Vogtmann2}, therefore $\dJ$ is homotopy equivalent to $\dK$ by Theorem~\ref{t:bg}.
\end{remark}

\subsection{The thickenings $\Lm$ and $\K$ are homotopy CM}

In \cite{Vogtmann}, Vogtmann showed that the spine of Outer space, and its reduced version, are homotopy CM complexes. We sketch below how her work shows that this is also true for $\Lm$ and $\K$. Very roughly, the link of a vertex $T$ in the spine decomposes as a join with collapses of $T$ on one side and blow-ups of $T$ on the other. Vogtmann has to carefully analyse both directions in her work;  we will see that the `blow-up' part of this analysis is also seen in links in $\Lm$ and $\K$, whereas the `descending' part of the link is considerably easier for us to analyse. Other than quoting different results from Vogtmann's paper, the proofs are the same for both the reduced and unreduced versions.

\begin{lemma}
If $p \in \Lm$ and $p = T_0 < T_1 < \cdots < T_k$ then $\hei_{\Lm}(p)=k$. Equivalently, height is the same in both $\Lm$ and $\FS'$. The same result holds for $\K$ and $\FSr'$.
\end{lemma}

\begin{proof}
The maximal chain-of-chains \[(T_1) < (T_{1}<T_2) < \cdots < (T_1< \cdots < T_k)\] in $\less{p}$ witnesses $\hei_{\Lm}(p)=k$.
\end{proof}

Hence we can read off height from chain length. As a consequence we will omit the subscript in the height function in the sequel.

\begin{proposition}\label{p:downwards_link}
    Let $p \in \Lm$. If $p \in \dL$ then $\less{p}$ is contractible, otherwise $\less{p}$ is homeomorphic to a $(\hei(p)-1)$--dimensional sphere.
    
      Let $p \in \K$. If $p \in \dK$ then $\less{p}$ is contractible, otherwise $\less{p}$ is homeomorphic to a $(\hei(p)-1)$--dimensional sphere.  
\end{proposition}

\begin{proof}
We give the proof for $\Lm$ and omit the proof for $\K$. Suppose $p \in \Lm$ and let $p = T_0 < T_1 < \cdots < T_k$. If $p \not \in \dL$, then $T_0$, and hence each blow-up $T_i$, belongs to $\X$. Hence every subchain of $p$ is an element of $\X$. It follows that $\less{p}$ is the poset of all proper subchains of $p$, so is isomorphic to the boundary of a $\hei(p)$--simplex. 
    
    If $p \in \dL$ then $T_0 \in \dF$. Let $T_j$ be the first element of the chain lying in $\X$. The subchains of $T_j<T_{j+1}< \cdots <T_k$ span a (full) simplex $\Delta$ in $\less{p}$. Furthermore, the monotone retraction $\chi:\Lm \to \Lm$ given in Proposition~\ref{p:retract} restricts to a retraction $r \colon \less{p} \to \Delta$, which is a homotopy equivalence by Corollary~\ref{c:quillen}. Hence $\less{p}$ is contractible.
\end{proof}

\begin{proposition}
Suppose $p,p' \in \Lm$ with $p<p'$. Then $(p,p')$ is homeomorphic to a $(\hei(p')-\hei(p)-2)$--dimensional sphere.

Suppose $p,p' \in \K$ with $p<p'$. Then $(p,p')$ is homeomorphic to a $(\hei(p')-\hei(p)-2)$--dimensional sphere.
\end{proposition}

\begin{proof}
Again we give the proof only for $\Lm$. Suppose $p,p' \in \Lm$. If $p'=T_0 < T_1 <\ldots < T_k$ then $p= T_{i_0} <T_{i_1} <  \ldots< T_{i_l}$, with $0\leq i_j \leq k$ for $1 \leq j \leq l$.  In this case, $\hei(p')-\hei(p)=k-l$. 

As there are $k-l$ elements of $p'$ that are not in $p$, the chains between $p$ and $p'$ are determined by proper subsets of these $k-l$ elements. As a poset this is isomorphic to the boundary of a $k-l-1$ simplex. Hence $(p,p')$ is homeomorphic to a sphere of dimension $k-l-2=\hei(p')-\hei(p)-2$.
\end{proof}

\begin{proposition}
    Suppose $p \in \Lm$. Then $p_<$ is homeomorphic to a wedge of spheres of dimension $3N-5-\hei(p)$.
    
   Suppose $p \in \K$. Then $p_<$ is homeomorphic to a wedge of spheres of dimension $3N-5-\hei(p)$.\end{proposition}
\begin{proof}

As before, we assume $p \in \Lm$ and will make it clear where the proof differs for $p \in \K$. Suppose $\hei(p)=k$, so that $p=T_0 < T_1 < \cdots < T_k$ for some sequence of trees in $\FS$.  As $p_<$ consists of superchains of $p$, it is isomorphic to the barycentric subdivision of
\begin{equation} \less{(T_0)} \ast (T_0,T_1) \ast (T_2,T_3) \ast \cdots \ast (T_k)_<, \label{e:join2}\end{equation} where these intervals are in $\FS$ (each factor in this join corresponds to a place where new elements can be added to the chain). Let $E_i$ denote the number edges in the graph $T_i/F_N$. As $\less{(T_0)}$ consists of trees made by collapsing some, but not all, of the edge orbits in $T_0$, it is homeomorphic to the boundary of an $E_0-1$ simplex, and is an ($E_0-2$)-sphere. Similarly, $(T_{i-1},T_{i})$ is isomorphic to the poset of proper subsets of the edge orbits in $T_{i}$ that are collapsed when passing to $T_{i-1}$, so is an $(E_i-E_{i-1}-2)$-sphere.

Let $V_k$ be the number of vertices of $T_k/F_N$. The poset $(T_k)_<$ is equal to the poset of blow-ups of $T_k/F_N$, which is homotopy equivalent to a wedge of spheres of dimension $2N-V_k-3$ (when $p \in \Lm$ this is shown in the proof of \cite[Corollary~2.7]{Vogtmann}, when $p \in \K$ this is \cite[Corollary~3.2]{Vogtmann}). Putting together all of the above using Equation~\eqref{e:join2}, and noting that there are $k+2$ terms in the join, we have shown that $\less{p}$ is homotopy equivalent to a wedge of spheres of dimension equal to\begin{align*} (E_0-2) +\sum_{i=1}^k(E_i-E_{i-1}-2) + &(2N-V_k-3) +(k+1)\\ &=2N +(E_k-V_k+1)-5-k\\ &= 3N-5-\hei(p). \end{align*}
Here we used the Euler characteristic formula $E-V=N-1$.
\end{proof}

\begin{theorem}\label{t:Jewel_space_is_CM}
The thickened spine $\Lm$ is a $3N-4$-dimensional homotopy CM complex. A simplex $\sigma \in \Lm$ is in the boundary $\dL$ if and only if the local homology of $\Lm$ vanishes at $\sigma$. 

The thickened reduced  spine $\K$ is a $3N-4$-dimensional homotopy CM complex. A simplex $\sigma \in \K$ is in the boundary $\dK$ if and only if the local homology of $\K$ vanishes at $\sigma$. 
\end{theorem}

\begin{proof}
We again omit the proof for $\K$. The homotopy CM condition for $\Lm$ follows from Proposition~\ref{p:homotopy_CM_conditions} and the work above. Suppose $\sigma=p_0<p_1< \cdots < p_k$ is a simplex in $\Lm$. The link of $\sigma$ decomposes as a join as described in Equation~\eqref{e:join}. In the work above, these join components are always nontrivial wedge products of spheres unless $p_0 \in \dL$, in which case $\less{(p_0)}$ is shown to be contractible in Proposition~\ref{p:downwards_link}. Therefore the link of $\sigma$ is homotopy equivalent to a nontrivial wedge of spheres unless $p_0 \in \dL$, when the link is contractible. As $\dL$ is upwards-closed in $\Lm$ the element $p_0 \in \dL$ if and only if $p_i \in \dL$ for all $i$, which is if and only if $\sigma \in \dL$. 
\end{proof}

\subsection{The dualizing module of $\out(F_N)$}

We are now in a position to prove Theorem~\ref{t:out} from the introduction, which for convenience we repeat below.

\begin{theorem}
Let $\X$ be the spine of Outer space and let $\Xr$ be the spine of reduced Outer space. The dualizing module of $\out(F_N)$ is isomorphic to the groups \begin{align*} H_0(\X,h^{2N-3})&\cong H^{2N-3}_c(\X,\mathbb{Z}) , \text{ and}\\
H_0(\Xr,h_r^{2N-3}) &\cong H^{2N-3}_c(\Xr,\mathbb{Z}) , \end{align*} where $h^{2N-3}$ and $h_r^{2N-3}$ are the degree $2N-3$ local cohomology cosheaves of $\X$ and $\Xr$, respectively. In contrast, the homology of the simplicial boundary of (reduced) Outer space satisfies
\begin{align*} H_{3N-5-k}(\dF;\mathbb{Z}) &\cong H^k_c(\X;h_*(\CV)|_{\X}),\text{ and}\\ H_{3N-5-k}(\dFr;\mathbb{Z}) &\cong H^k_c(\X;h_*(\CVr)|_{\Xr}), \end{align*} where $h_*(\CV)|_{\X}$ and and $h_*(\CVr)|_{\Xr}$  are the restrictions of the local homology sheaf of Outer space (and reduced Outer space) to the spine (and its reduced version, respectively).\end{theorem}

\begin{proof}
As $\out(F_N)$ is a virtual duality group of dimension $2N-3$ and and the (reduced) spine is a $2N-3$--dimensional CM complex by \cite{Vogtmann}, the first pair of isomorphisms follow from Corollary~\ref{c:homological_D} in the introduction. By Theorem~\ref{t:Jewel_space_is_CM}, the thickened spine $\Lm$ is homotopy CM. Furthermore, every vertex of $\Lm$ either belongs to $\X'$ or $\dL$. Hence we can apply the relative CM duality of Section~\ref{s:relative_CM} with $X=\Lm$, $L=\X'$ and $L^\textnormal{vc}=\dL$. The first isomorphism in the relative CM duality theorem gives \[H^k_c(\X;h_*(\mathcal{\Lm})|_{\X}) \cong H_{3N-4-k}(\Lm,\dL;\mathbb{Z}).\] The interior of the thickened spine is homeomorphic to Outer space, so $h_*(\mathcal{\Lm})|_{\X}=h_*(\CV)|_{\X}$. As $\Lm$ is contractible, $H_{3N-4-k}(\Lm,\dL;\mathbb{Z})\cong H_{3N-5-k}(\dL;\mathbb{Z})$ by the long exact sequence in homology. Finally $\dL$ is homotopy equivalent to $\dF$ by Theorem~\ref{t:bg}. The same argument goes through for the final isomorphism, using $\K$ in place of $\Lm$.
\end{proof}

\begin{remark}[Related Work]
In \cite{LV}, Lazarev and Voronov use Verdier duality along with similar sheaf-theoretic methods to reprove results of Kontsevich relating graph cohomology to the rational cohomology of $\out(F_N)$. There, CM duality is manifested via the Verdier dual sheaf associated to a cyclic operad $\mathcal{O}$ being concentrated in a single degree (see \cite[Corollary~3.12]{LV} and its preceding paragraph).  Similar statements can be made by applying CM duality to the quotient $Y_N=\X/\out(F_N)$ of the spine. Indeed, the chain complex that computes homology of $Y_N$ with coefficients in the local cohomology cosheaf shares a lot in common with the \emph{forested graph complex} introduced by Conant and Vogtmann in their paper on Kontsevich's work \cite{ConantV}. These previous papers focus on the quotient $Y_N$, whereas we have been focusing on the spine itself.

In \cite{Himes}, the authors show that the homology of the free factor complex for $\aut(F_N)$ is not the same as the dualizing module of $\aut(F_N)$ when $N=5$. The free factor complex is expected to have a different homotopy type to the bordifications studied above (this is discussed in more detail in \cite{BG}).
\end{remark}

\section{Bounding the top-dimensional cohomology of $\out(F_N)$}\label{s:bound}

 In this section, we see how cosheaf homology gives a resolution for the dualizing module of $\out(F_N)$, and use this to prove the following:

\begin{theorem}\label{t:top_dim_bound}
Let $\rho \in \Xr$ be a rose in the spine of reduced outer space, and let $\stab(\rho) <\out(F_N)$ be the stabilizer of this rose. Then \[ \dim(H^{2N-3}(\out(F_N);\mathbb{Q})) \leq \dim((h^{2N-3}(\rho)\otimes \mathbb{Q})_{\stab(\rho)}),\]
where $h^{2N-3}(\rho)$ is the local cohomology of the reduced spine at this rose.
\end{theorem}

Recall that a \emph{rose} is a point $T$ in the spine such that $T/F_N$ is a graph with one vertex and $N$ loops, and we use $M_G$ to denote the group of coinvariants of a $G$-module $M$. As $\out(F_N)$ is a virtual duality group, it is a  duality group of dimension $2N-3$ over $\mathbb{Q}$ \cite[Theorem~9.9]{Bierinotes}. In particular this implies:
$$
H^{2N-3}(\out(F_N);\mathbb{Q}) \cong H_0(\out(F_N);D \otimes \mathbb{Q})=(D \otimes \mathbb{Q})_{\out(F_N)},
$$
where $D$ is the dualizing module of $\out(F_N)$.
To prove Theorem~\ref{t:top_dim_bound}, we will show that with integer coefficients one has
\begin{equation}\label{e:Dinvsubloccoh}
    D_{\out(F_N)} \subseteq h^{2N-3}(\rho)_{\stab(\rho)},
\end{equation}
which implies the dimensional inequality after passing to rational coefficients.
As above, the dualizing module of $\out(F_N)$ is  \[ D=H^{2N-3}_c(\Xr;\mathbb{Z}). \] As we are in the top dimension, the coboundary map is trivial so $H^{2N-3}_c(\Xr;\mathbb{Z})$ is generated by cohomology classes of the form $[\alpha^*]$, as $\alpha$ varies over the top-dimensional simplices of $\Xr$. Viewing $G=\out(F_N)$ as acting on $\Xr$ on the left, this gives $H^{2N-3}_c(\Xr;\mathbb{Z})$ the usual right $G$-module structure by \[ [\alpha^*]g=[g^{-1}(\alpha)^*].\]By \cite{Vogtmann}, the spine is locally Cohen--Macaulay, so we can apply CM duality (Theorem~\ref{t:dualitytheorem}) to obtain an isomorphism \[[\Xr]\cap - \colon H^{k}_c(\Xr;\mathbb{Z}) \xrightarrow{\cong} H_{2N-3-k}(\Xr;h^{2N-3}), \] where $h^{2N-3}$ is the local cohomology cosheaf of the spine and $[\Xr] \in H^{\text{lf}}_{2N-3}(\Xr;h^{2N-3})$ is the fundamental class. As $H^*_c(\Xr;\mathbb{Z})$ is concentrated in degree $2N-3$,  $H_*(\Xr;h^{2N-3})$ is concentrated in degree zero, so its chain complex is acyclic. This gives a resolution
\begin{equation}\label{e:resolutionofD}
    \cdots \to C_1(\Xr;h^{2N-3}) \to C_0(\Xr;h^{2N-3}) \to D\to 0
\end{equation} of the dualizing module. The modules in this chain complex are (see Equation~\eqref{eq:cpltlysuppsheafcoh}) \[ C_k(\Xr;h^{2N-3})=\oplus_{\sigma \in (\Xr)_k} h^{2N-3}(\sigma), \] where $h^{2N-3}(\sigma)=H^{2N-3}(\Xr,\Xr-\ostar(\sigma))$ is the local homology at $\sigma$. 
To distinguish where local homology classes live, for $\phi \in h^{2N-3}(\sigma)$ we use $\sigma \under \phi$ to denote the corresponding element of $C_n(\Xr;h^{2N-3})$ in the $h^{2N-3}(\sigma)$-summand. 
% In this language, the chain complex also inherits a right $G$-action via \[ (\sigma \under \phi)g = g^{-1}(\sigma) \under g^*(\phi),  \]
% where $g^*:H^*(\Xr;\Xr-\ostar(\sigma)) \to H^*(\Xr;\Xr-\ostar(g^{-1}(\sigma)))$ is the map on local cohomology groups induced by the map of pairs \[(\Xr,\Xr-\ostar(g^{-1}(\sigma)) \to (\Xr;\Xr-\ostar(\sigma))\] given by $g$.

At the chain level, the cap product with the fundamental class $[\Xr]$ given in \cite{WW1} is the $G$-equivariant (see \cite[Section 5]{WW1}) map \[ c_{2N-3}\colon C^{2N-3}(\Xr;\mathbb{Z}) \to C_0(\Xr;h^{2N-3})\] given by \begin{equation} \label{e:cap_on_maximal_simplex}  c_{2N-3}(\alpha^*)=\alpha_0 \under [\alpha^*],\end{equation} where $\alpha_0$ is the first vertex of $\alpha$ (using the ordering/orientation on $\Xr$ induced by the poset structure), and $\alpha_0 \under [\alpha^*]$ lives in the local cohomology at the vertex $\alpha_0$. The quickest way to see this is via Equation~(14) on page 13 of \cite{WW1}, noting that the fundamental class $[\Xr]$ is represented by the infinite sum $\sum{\alpha \under [\alpha^*]}$, as $\alpha$ ranges over all maximal simplices. In the poset structure we use on the spine every maximal simplex $\alpha$ is given by a chain \[ T_0 < T_1 < \cdots < T_{2N-3},\] where the initial vertex $T_0$ of $\alpha$ is a rose. We can leverage this to restrict the image of $c_{2N-3}$:

\begin{prop}\label{p:classes_have_a_rep_in_R}
Let $R$ be the submodule of $C_0(\Xr;h^{2N-3})$ spanned by the local cohomology groups of the roses in $\Xr$. Every element of $D=C_0(\Xr;\mathbb{Z})/\im{d_1}$ has a representative in $R$ .   
\end{prop}

\begin{proof}
The cap product $c_{2N-3}$ induces an isomorphism between $H^{2N-3}_c(\Xr;\mathbb{Z})$ and $H_0(\Xr;h^{2N-3})$. The image of $c_{2N-3}$ is contained in $R$ by \eqref{e:cap_on_maximal_simplex}, as the first vertex $\alpha_0$ of any maximal simplex $\alpha$ is a rose. In particular, $H_0(\Xr;h^{2N-3})$ is generated by the classes in $R$.
\end{proof}

\begin{corollary}\label{c:dualizing_module_as_quotient_of_R}
Let $R$ be the submodule of $C_0(\Xr;h^{2N-3})$ spanned by the local cohomology groups of the roses in $\Xr$, and let \[d_1:C_1(\Xr;h^{2N-3})\to C_0(\Xr;h^{2N-3})\] be the boundary map in cosheaf homology. Then \[ D \cong R/(\im d_1 \cap R).\]
\end{corollary}

\begin{proof}
By CM dualiy, \[ D\cong H_0(\Xr;h^{2N-3})\cong C_0(\Xr;h^{2N-3})/\im d_1, \] and the inclusion $R \subset C_0(\Xr;h^{2N-3})$ induces an injection \[ R/(\im d_1 \cap R) \hookrightarrow C_0(\Xr;h^{2N-3})/\im d_1 \cong D.\] By Proposition \ref{p:classes_have_a_rep_in_R}, every element of $D$ has a representative in $R$, so this map is also surjective.
\end{proof}
    
\begin{proposition}
Let $R$ be as above and let $I$ be the augmentation ideal of the group ring of $\out(F_N)$. If $\rho$ is a rose in $\Xr$ and $\stab(\rho)$ is its stabilizer in $\out(F_N)$ then 
\[ R/RI =R_{\out(F_N)} \cong h^{2N-3}(\rho)_{\stab(\rho)}.\]
\end{proposition}

\begin{proof}
As $R$ decomposes as a direct summand over the $h^{2N-3}(\rho)$ and $\out(F_N)$ acts transitively on the summands, it follows that \[ R=\Ind^{\out(F_N)}_{\stab(\rho)}h^{2N-3}(\rho), \] i.e. $R$ is the induced representation of the $\stab(\rho)$ representation $h^{2N-3}(\rho)$ (see e.g. Proposition III.5.3 in \cite{Brown}). Then \[H_*(\out(F_N);R)\cong H_*(\stab(\rho);h^{2N-3}(\rho))\] by Shapiro's lemma \cite[Proposition~III.6.2]{Brown}. The result follows by taking $*=0$.
\end{proof}

\begin{proof}[Proof of Theorem~\ref{t:top_dim_bound}]
By Corollary~\ref{c:dualizing_module_as_quotient_of_R}, we have \[ D\cong \frac{R}{\im d_1 \cap R},\] therefore the coinvariants of this module are given by 
\[D_{\out(F_N)} \cong \frac{R}{\im d_1 \cap R + RI}\subset \frac{R}{RI}\cong h^{2N-3}(\rho)_{\stab(\rho)}  \]
Theorem~\ref{t:top_dim_bound} is then a consequence of $\out(F_N)$ being a duality group over $\mathbb{Q}$ with rational dualizing module $D \otimes \mathbb{Q}$.
\end{proof}

\begin{remark}
We finish this section with some remarks.
\begin{itemize}
\item To show that the inclusion from Equation~\eqref{e:Dinvsubloccoh} is actually an isomorphism, it follows from the proof that it is enough to show that $RI$ contains  $\im d_1 \cap R$. This is definitely false in the unreduced case, and seems unlikely to be true even in the reduced setting.
\item The representation theory of the $\stab(\rho)$ action on the cohomology of the link (or equivalently, the local cohomology) is well understood in terms of \emph{tree representations} of the symmetric group. See \cite{Barcelo} and  \cite{RWtree}. For instance, in the unreduced setting the dimension of $h^{2N-3}(\rho)$ is $(2N-2)!$ However, the structure of the link is much less well-understood in the reduced spine.
\item While we suspect it may be possible to prove Theorem~\ref{t:top_dim_bound} without using CM duality, the theorem was only formulated after a careful consideration of the the resolution of the dualizing module obtained in Equation~\eqref{e:resolutionofD}. Furthermore, we found that the cap product in CM duality facilitated a particularly clean proof of Proposition~\ref{p:classes_have_a_rep_in_R}.
\end{itemize}
\end{remark}

\section{2-complexes, visible irreducibility and the local homology sheaf} \label{s:examples2}

In dimension two, links of points are homeomorphic to graphs, and we will see that the local Cohen--Macaulay condition is equivalent to the graphs that appear as vertex links being connected and nondegenerate (not a single point). A related condition called \emph{visible irreducibility}, appears in work of Louder and Wilton \cite{LW} (reformulated as below in \cite[Remark~4.2]{Wilton}) . Informally this says that there are no obvious local simplifications one can perform on the 2-complex. More precisely, a complex is \emph{visibly irreducible} if:

\begin{enumerate}[(i)]
    \item every vertex link contains at least one edge;
    \item  every vertex link is connected;
    \item vertex links do not contain leaves; and
    \item vertex links do not have cut vertices.
\end{enumerate}

Otherwise, a complex is \emph{visibly reducible}. We now turn to Theorem~\ref{t:visible_irreducibility} from the introduction, which is restated below.

\begin{theorem}
Let $Y$ be a simplicial 2-complex. Then $Y$ is a local homology CM complex if and only if it satisfies (i) and (ii). Furthermore, the following are equivalent:
\begin{itemize}
\item $Y$ is visibly irreducible.
\item $Y$ is a local homology CM complex and the restriction maps of the local homology sheaf are surjective.
\item $Y$ is a local homology CM complex and the corestriction maps of the local cohomology cosheaf are injective.
\end{itemize}
\end{theorem}

\begin{proof}
We first show that if $Y$ is locally CM then it satisfies (i) and (ii). If $Y$ is locally CM then it is pure (all of its maximal simplices are 2-simplices), so in particular every vertex is contained in a 2-simplex and (i) holds. To see that (ii) also holds, let $V \in Y$ be a vertex, and let $\lk_Y(V)$ be the vertex link. By comparing the chain complexes $C_{\bullet}^V(Y) $ (as defined in Section~\ref{s:local_hom}) and $\tilde{C}_{\bullet-1}(\lk_Y(V))$ (that computes the reduced homology of the link) we see that $h_k(V)\cong \tilde{H}_{k-1}(\lk_Y(V))$: every adjacent edge $E$ at $V$ determines a vertex $v_E \in \lk_Y(V)$ and 2-simplices containing $V$ give edges in $\lk_Y(V)$. If (ii) does not hold, there exists a vertex $V$ such that $\lk_Y(V)$ is a disconnected graph, so that $h_1(V)=\tilde{H}_{0}(\lk_Y(V))$ is nontrivial. So $Y$ being locally CM also implies (ii).

Now assume that $Y$ satisfies (i) and (ii). We want to to show that the local homology $h_*(\sigma)$ is concentrated in degree 2 for every simplex $\sigma\in Y$. We start with the vertices: by (i) and (ii) each vertex link is a non-empty connected graph with at least one edge, so $\tilde{H}_*(\lk_Y(V)) $ is concentrated in degree 1, and therefore $h_*(V) \cong \tilde{H}_{*-1}(\lk_Y(V))$ is concentrated in degree 2. Moving on to the edges, $h_*(E)$ is concentrated in degree two if and only if $E$ is contained in a face. If this is not the case, $E$ will be an isolated vertex in each of its vertex links, contradicting (ii) (as there is at least one edge in each vertex link by (i)). The local homology at each 2-simplex $\sigma$ satisfies $h_2(\sigma)=\mathbb{Z}$ and $h_k(\sigma)=0$ for $k \neq 2$ by maximality. Putting all of the above together, $Y$ is a local homology CM complex if and only if the vertex links of $Y$ satisfy (i) and (ii).

We now move on the equivalence of the bullet points. As there is no torsion in local (co)homology when the complex is locally CM, equivalence of the second and third bullet points follows from the universal coefficient theorem. Therefore we will focus on the local homology sheaf. Firstly, if $F$ is a 2--simplex in $Y$ and $E$ is an edge of $F$, the map $h_2(E) \to h_2(F)$ is surjective if and only if the local homology $h_2(E)$ does not vanish. This happens if and only if $E$ is contained in at least two 2--simplices. This in turn is equivalent to vertex links not containing leaves.  

Now suppose that $Y$ satisfies (i), (ii) and (iii). As above let $V$ be a vertex of an edge $E \in Y$.  The local homology in $Y$ at $E$ only considers 2--simplices containing $E$, so in terms of the notation above we have \[ h_2(E)\cong H_1(\Gamma,\Gamma-\ostar_\Gamma(v_E)), \] where $\ostar_\Gamma(v_E)$ is the open star of $v_E$ in the vertex link $\Gamma$. Furthermore, one can check that the map $h_2(V)\to h_2(E)$ is the same as the map \[H_1(\Gamma) \to H_1(\Gamma,\Gamma-\ostar_{\Gamma}(v_E)) \] induced by the inclusion of pairs $(\Gamma,\emptyset) \to (\Gamma,\Gamma-\ostar_\Gamma(v_E))$. By considering the long exact sequence in homology, the cokernel of $h_2(E) \to h_2(V)$ is then isomorphic to \[ \tilde{H}_0(\Gamma-\ostar_\Gamma(v_E)) \] 
Hence the map $h_2(V) \to h_2(E)$ is surjective unless $\Gamma-\ostar_\Gamma(v_E)$ has more than one component, or equivalently unless $v_E$ is a cut vertex of $\lk_Y(V)=\Gamma$.
\end{proof}

\bibliography{lib}

\newcommand{\etalchar}[1]{$^{#1}$}
\begin{thebibliography}{HMNP23}
\expandafter\ifx\csname url\endcsname\relax
  \def\url#1{\texttt{#1}}\fi
\expandafter\ifx\csname doi\endcsname\relax
  \def\doi#1{\burlalt{doi:#1}{http://dx.doi.org/#1}}\fi
\expandafter\ifx\csname urlprefix\endcsname\relax\def\urlprefix{URL }\fi
\expandafter\ifx\csname href\endcsname\relax
  \def\href#1#2{#2}\fi
\expandafter\ifx\csname burlalt\endcsname\relax
  \def\burlalt#1#2{\href{#2}{#1}}\fi

\bibitem[Ata14]{Atanasov}
R.~Atanasov.
\newblock Geometric two-dimensional duality groups.
\newblock {\em Groups Geom. Dyn.}, 8(1):69--95, 2014.
\newblock \doi{10.4171/GGD/217}.

\bibitem[Bar90]{Barcelo}
H.~Barcelo.
\newblock On the action of the symmetric group on the free {L}ie algebra and
  the partition lattice.
\newblock {\em J. Combin. Theory Ser. A}, 55(1):93--129, 1990.
\newblock \doi{10.1016/0097-3165(90)90050-7}.

\bibitem[Bar16]{Bartholdi}
L.~Bartholdi.
\newblock The rational homology of the outer automorphism group of {$F_7$}.
\newblock {\em New York J. Math.}, 22:191--197, 2016.
\newblock \urlprefix\url{http://nyjm.albany.edu:8000/j/2016/22_191.html}.

\bibitem[BE73]{BE}
R.~Bieri and B.~Eckmann.
\newblock Groups with homological duality generalizing {P}oincar\'{e} duality.
\newblock {\em Invent. Math.}, 20:103--124, 1973.
\newblock \doi{10.1007/BF01404060}.

\bibitem[BF00]{BF}
M.~Bestvina and M.~Feighn.
\newblock The topology at infinity of {${\rm Out}(F_n)$}.
\newblock {\em Invent. Math.}, 140(3):651--692, 2000.
\newblock \doi{10.1007/s002220000068}.

\bibitem[BG20]{BG}
B.~Br\"{u}ck and R.~Gupta.
\newblock Homotopy type of the complex of free factors of a free group.
\newblock {\em Proc. Lond. Math. Soc. (3)}, 121(6):1737--1765, 2020.
\newblock \doi{10.1112/plms.12381}.

\bibitem[Bie81]{Bierinotes}
R.~Bieri.
\newblock {\em Homological dimension of discrete groups}.
\newblock Queen Mary College Mathematics Notes. Queen Mary College, Department
  of Pure Mathematics, London, second edition, 1981.

\bibitem[Bj{\"{o}}16]{Bjoerner}
A.~Bj{\"{o}}rner.
\newblock ``{L}et {$\Delta$} be a {C}ohen-{M}acaulay complex {$\dots$}''.
\newblock In {\em The mathematical legacy of {R}ichard {P}. {S}tanley}, pages
  105--118. Amer. Math. Soc., Providence, RI, 2016.
\newblock \doi{10.1090//mbk/100/06}.

\bibitem[Bj{\"{o}}95]{Bjo2}
A.~Bj{\"{o}}rner.
\newblock Topological methods.
\newblock In {\em Handbook of combinatorics, {V}ol. 1, 2}, pages 1819--1872.
  Elsevier Sci. B. V., Amsterdam, 1995.

\bibitem[BM01]{BM}
N.~Brady and J.~Meier.
\newblock Connectivity at infinity for right angled {A}rtin groups.
\newblock {\em Trans. Amer. Math. Soc.}, 353(1):117--132, 2001.
\newblock \doi{10.1090/S0002-9947-00-02506-X}.

\bibitem[BMP{\etalchar{+}}24]{BMPSW}
B.~Br\"{u}ck, J.~Miller, P.~Patzt, R.~J. Sroka, and J.~C.~H. Wilson.
\newblock On the codimension-two cohomology of {${\rm SL}_n(\Bbb{Z})$}.
\newblock {\em Adv. Math.}, 451:Paper No. 109795, 82, 2024.
\newblock \doi{10.1016/j.aim.2024.109795}.

\bibitem[BPS23]{BPS}
B.~Brück, P.~Patzt, and R.~J. Sroka.
\newblock A presentation of symplectic {S}teinberg modules and cohomology of
  $\operatorname{Sp}_{2n}(\mathbb{Z})$.
\newblock {\em arXiv:2306.03180}, 2023.

\bibitem[Bro82]{Brown}
K.~S. Brown.
\newblock {\em Cohomology of groups}, volume~87 of {\em Graduate Texts in
  Mathematics}.
\newblock Springer-Verlag, New York-Berlin, 1982.

\bibitem[Bro12]{Broaddus}
N.~Broaddus.
\newblock Homology of the curve complex and the {S}teinberg module of the
  mapping class group.
\newblock {\em Duke Math. J.}, 161(10):1943--1969, 2012.
\newblock \doi{10.1215/00127094-1645634}.

\bibitem[BS73]{BS}
A.~Borel and J.-P. Serre.
\newblock Corners and arithmetic groups.
\newblock {\em Comment. Math. Helv.}, 48:436--491, 1973.
\newblock \doi{10.1007/BF02566134}.

\bibitem[BSV18]{BSV}
K.-U. Bux, P.~Smillie, and K.~Vogtmann.
\newblock On the bordification of outer space.
\newblock {\em J. Lond. Math. Soc. (2)}, 98(1):12--34, 2018.
\newblock \doi{10.1112/jlms.12124}.

\bibitem[CFP12]{CFP2}
T.~Church, B.~Farb, and A.~Putman.
\newblock The rational cohomology of the mapping class group vanishes in its
  virtual cohomological dimension.
\newblock {\em Int. Math. Res. Not. IMRN}, (21):5025--5030, 2012.
\newblock \doi{10.1093/imrn/rnr208}.

\bibitem[CFP19]{CFP1}
T.~Church, B.~Farb, and A.~Putman.
\newblock Integrality in the {S}teinberg module and the top-dimensional
  cohomology of {${\rm SL}_n\mathcal{O}_K$}.
\newblock {\em Amer. J. Math.}, 141(5):1375--1419, 2019.
\newblock \doi{10.1353/ajm.2019.0036}.

\bibitem[CGP21]{CGP}
M.~Chan, S.~r. Galatius, and S.~Payne.
\newblock Tropical curves, graph complexes, and top weight cohomology of
  {$\mathcal{M}_g$}.
\newblock {\em J. Amer. Math. Soc.}, 34(2):565--594, 2021.
\newblock \doi{10.1090/jams/965}.

\bibitem[CP17]{CP}
T.~Church and A.~Putman.
\newblock The codimension-one cohomology of {${\rm SL}_n\Bbb Z$}.
\newblock {\em Geom. Topol.}, 21(2):999--1032, 2017.
\newblock \doi{10.2140/gt.2017.21.999}.

\bibitem[CV86]{CV}
M.~Culler and K.~Vogtmann.
\newblock Moduli of graphs and automorphisms of free groups.
\newblock {\em Invent. Math.}, 84(1):91--119, 1986.
\newblock \doi{10.1007/BF01388734}.

\bibitem[CV03]{ConantV}
J.~Conant and K.~Vogtmann.
\newblock On a theorem of {K}ontsevich.
\newblock {\em Algebr. Geom. Topol.}, 3:1167--1224, 2003.
\newblock \doi{10.2140/agt.2003.3.1167}.

\bibitem[Dav98]{Davis1998}
M.~W. Davis.
\newblock The cohomology of a {C}oxeter group with group ring coefficients.
\newblock {\em Duke Math. J.}, 91(2):297--314, 1998.
\newblock \doi{10.1215/S0012-7094-98-09113-X}.

\bibitem[Dav00]{Davis2000}
M.~W. Davis.
\newblock Poincar\'{e} duality groups.
\newblock In {\em Surveys on surgery theory, {V}ol. 1}, volume 145 of {\em Ann.
  of Math. Stud.}, pages 167--193. Princeton Univ. Press, Princeton, NJ, 2000.

\bibitem[Fow12]{Fowler}
J.~Fowler.
\newblock Finiteness properties for some rational {P}oincar\'{e} duality
  groups.
\newblock {\em Illinois J. Math.}, 56(2):281--299, 2012.
\newblock \urlprefix\url{http://projecteuclid.org/euclid.ijm/1385129948}.

\bibitem[Geo08]{Geoghegan}
R.~Geoghegan.
\newblock {\em Topological methods in group theory}, volume 243 of {\em
  Graduate Texts in Mathematics}.
\newblock Springer, New York, 2008.
\newblock \doi{10.1007/978-0-387-74614-2}.

\bibitem[Har86]{Harer}
J.~L. Harer.
\newblock The virtual cohomological dimension of the mapping class group of an
  orientable surface.
\newblock {\em Invent. Math.}, 84(1):157--176, 1986.
\newblock \doi{10.1007/BF01388737}.

\bibitem[HMNP23]{Himes}
Z.~Himes, J.~Miller, S.~Nariman, and A.~Putman.
\newblock The free factor complex and the dualizing module for the automorphism
  group of a free group.
\newblock {\em Int. Math. Res. Not. IMRN}, (22):19020--19068, 2023.
\newblock \doi{10.1093/imrn/rnac330}.

\bibitem[Iva87]{Ivanov}
N.~V. Ivanov.
\newblock Complexes of curves and the teichmüller modular group.
\newblock {\em Russian Mathematical Surveys}, 42(3):55, jun 1987.
\newblock \doi{10.1070/RM1987v042n03ABEH001440}.

\bibitem[LS76]{LS}
R.~Lee and R.~H. Szczarba.
\newblock On the homology and cohomology of congruence subgroups.
\newblock {\em Invent. Math.}, 33(1):15--53, 1976.
\newblock \doi{10.1007/BF01425503}.

\bibitem[LV08]{LV}
A.~Lazarev and A.~A. Voronov.
\newblock Graph homology: {K}oszul and {V}erdier duality.
\newblock {\em Adv. Math.}, 218(6):1878--1894, 2008.
\newblock \doi{10.1016/j.aim.2008.03.022}.

\bibitem[LW24]{LW}
L.~Louder and H.~Wilton.
\newblock Uniform negative immersions and the coherence of one-relator groups.
\newblock {\em Invent. Math.}, 236(2):673--712, 2024.
\newblock \doi{10.1007/s00222-024-01246-4}.

\bibitem[MPP21]{MPP}
J.~Miller, P.~Patzt, and A.~Putman.
\newblock On the top-dimensional cohomology groups of congruence subgroups of
  {${\rm SL}(n, \Bbb Z)$}.
\newblock {\em Geom. Topol.}, 25(2):999--1058, 2021.
\newblock \doi{10.2140/gt.2021.25.999}.

\bibitem[PS22]{PS}
A.~Putman and D.~Studenmund.
\newblock The dualizing module and top-dimensional cohomology group of {${\rm
  GL}_n(\mathcal{O})$}.
\newblock {\em Math. Z.}, 300(1):1--31, 2022.
\newblock \doi{10.1007/s00209-021-02769-9}.

\bibitem[Qui78]{Quillen}
D.~Quillen.
\newblock Homotopy properties of the poset of nontrivial {$p$}-subgroups of a
  group.
\newblock {\em Adv. in Math.}, 28(2):101--128, 1978.
\newblock \doi{10.1016/0001-8708(78)90058-0}.

\bibitem[RW96]{RWtree}
A.~Robinson and S.~Whitehouse.
\newblock The tree representation of {$\Sigma_{n+1}$}.
\newblock {\em J. Pure Appl. Algebra}, 111(1-3):245--253, 1996.
\newblock \doi{10.1016/0022-4049(95)00116-6}.

\bibitem[Vog90]{Vogtmann}
K.~Vogtmann.
\newblock Local structure of some {${\rm Out}(F_n)$}-complexes.
\newblock {\em Proc. Edinburgh Math. Soc. (2)}, 33(3):367--379, 1990.
\newblock \doi{10.1017/S0013091500004818}.

\bibitem[Vog24]{Vogtmann2}
K.~Vogtmann.
\newblock The boundary of bordified outer space.
\newblock {\em arXiv:2404.11371}, 2024.

\bibitem[Wal81]{Walker}
J.~W. Walker.
\newblock Homotopy type and {E}uler characteristic of partially ordered sets.
\newblock {\em European J. Combin.}, 2(4):373--384, 1981.
\newblock \doi{10.1016/S0195-6698(81)80045-5}.

\bibitem[Wal79]{Wall}
C.~Wall~(ed).
\newblock List of problems.
\newblock In {\em Homological Group Theory}, London Math. Soc. Lecture Notes
  36. Cambridge Univ. Press, 1979.

\bibitem[Wil24]{Wilton}
H.~Wilton.
\newblock Rational curvature invariants for 2-complexes.
\newblock {\em Proc. A.}, 480(2296):Paper No. 20240025, 39, 2024.

\bibitem[WW23]{WW1}
R.~D. Wade and T.~A. Wasserman.
\newblock Duality for {C}ohen-{M}acaulay complexes through combinatorial
  sheaves.
\newblock 2023.

\end{thebibliography}
\bibliographystyle{halpha-abbrv}

\end{document}